\documentclass[a4paper,12pt]{amsart}

\usepackage{amsmath,latexsym,amssymb,mathabx}
\usepackage[german,english]{babel}
\usepackage[applemac]{inputenc}
\usepackage{graphicx}

\usepackage{xcolor}

\newtheorem{thm}{Theorem}[section]
\newtheorem{prop}[thm]{Proposition}
\newtheorem{lemma}[thm]{Lemma}
\newtheorem{cor}[thm]{Corollary}

\newenvironment{defi}
{\stepcounter{thm}\par\medskip
   \noindent \textbf{Definition~\thesection.\arabic{thm}} \rmfamily}{\medskip}

\newenvironment{rem}
{\stepcounter{thm}\par\medskip
   \noindent \textbf{Remark~\thesection.\arabic{thm}} \rmfamily}{\medskip}

\newenvironment{exa}
{\stepcounter{thm}\par\medskip
   \noindent \textbf{Example~\thesection.\arabic{thm}} \rmfamily}{\medskip}

\newcommand{\NN}{{\mathbb N}}
\newcommand{\ZZ}{{\mathbb Z}}

\newcommand{\QQ}{{\mathbb Q}}

\newcommand{\del}{\partial}
\newcommand{\ON}{\operatorname}

\renewcommand{\a}{\alpha}

\newcommand{\G}{\Gamma}

\newcommand{\tto}{\longrightarrow}

\newcommand{\inv}{^{^{-1}}}

\newcommand{\cfami}{{\mathcal C}}

\newcommand{\hfami}{{\mathcal H}}

\newcommand{\mfami}{{\mathcal M}}

\newcommand{\qfami}{{\mathcal Q}}
\newcommand{\sfami}{{\mathcal S}}

\newcommand{\cutoff}[1]{}

\newcommand{\br}{\ON{Br}}

\newcommand{\ord}{\ON{ord}}

\newcommand{\map}{\mathcal M\!\!\;{ap}}

\newcommand{\fracc}[2]{#1 \big/\!\raisebox{-.8mm}{$#2$}}

\begin{document}


\title[Branch stabilisation\hfill]{Branch 
stabilisation for the components of\\ Hurwitz moduli spaces 
of Galois covers}

\author{Michael L\"onne}

\date{\today}

\begin{abstract}
We consider components of Hurwitz
moduli space of $G$-Galois covers and set up a powerful
algebraic framework to study the set of corresponding
equivalence classes of monodromy maps.
Within that we study geometric stabilisation by various
$G$-covers branched over the disc.

Our results addresses the problem to decide equivalence
and stable equivalence algebraically. We recover a homological
invariant, which we show to distinguish the equivalence classes of
given boundary monodromy and Nielsen type, if the latter
is sufficiently large in the appropriate sense. 
\end{abstract}

\thanks{The present work was done in the framework  of the 
ERC Advanced grant n. 340258, `TADMICAMT' }

\maketitle

%
%
%
%
%

\section{Introduction}

In this article we want to discuss at length some group theoretical aspects
crucial to the study of connected components of Hurwitz spaces
of coverings, which can be traced back to \cite{Clebsch,Hurwitz}.

In the narrow sense, the objective is the study of morphisms 
$p\colon C \to C/G=: C'$ induced by an
effective action of a finite group $G$. The existence of such an effective 
action provides information about the complex structure of the curve and
also about the group. In the broader sense, we include branched
covers over Riemann surfaces, usually with at least one boundary
component.

Let us recall first that the geometry of the covering $p$ encodes several 
numerical invariants that are constant under deformation:
the genus $g'$ of the base $C'$ and its number of boundary components, 
the number $d$ of branch points 
$y_1,\dots, y_d\in C'$ and the orders $m_1,\dots, m_d$ of the local 
monodromies - strictly speaking as an unordered multi-set. 
These invariants form the {\it primary numerical type}
and have been the object of intensive studies since long. 

A refined invariant is obtained from the monodromy
$\mu:\pi_1(C'\setminus \{ y_1, \dots , y_d\})\to G$ of the regular unramified
cover given by restriction of $p$ to the complement of the branch points, 
$p^{-1}(C'\setminus \{ y_1, \dots , y_d \})$. 
Instead of keeping track only of the orders $m_i$ of the elements in $G$ 
associated to the local monodromies at the punctures $y_i$, we consider 
the multi-set of their conjugacy classes. 
This is given by the \emph{Nielsen type}, cf. \cite{nielsen},
the class function $\nu$ which on each 
conjugacy class $\mathcal{C}$ in $G$ takes the cardinality of local
monodromies in $\mathcal{C}$ as its value.
The function $\nu$ can be characterised also without taking recourse
to the quotient $C'$. For each conjugacy class $\cfami$ just count the 
number of $G$-orbits on $C$ such that an element is in $\cfami$ which
rotates a disc around some point of the orbit by the smallest possible
angle.

While many groups have been shown to allow classification by primary 
numerical type and Nielsen type, we are interested into arbitrary groups
and get the motivation by the progress in the case of genus stabilisation.

Let us briefly recall the main results in the case of free actions. There
the second homology group $H_2(G)$ was shown
\begin{itemize}
\item
to classify equivalence classes of unbranched $G$-coverings for 
abelian and metabelian groups \cite{Edm I,Edm II}.
\item
to classify stable equivalence classes for every group \cite{Livingston}.
\item
to classify equivalence classes for every group if the genus $g'$ is
sufficiently large \cite{DT}.
\end{itemize}
An analogous result to the last in the case of non-free action was
proved with F.\ Catanese and F.\ Perroni \cite{CLP16}, where the
second homology group had to be replaced by a quotient $H_{2,\Gamma}$.

The case of branch stabilisation is more involved. While genus stabilisation
corresponds to connected sum with a trivial $G$-cover over the torus,
summing a branched cover does involve making non-trivial choices for the
monodromies at the branch points.
In fact we will stabilise by boundary connected sum with a punctured disc, 
but consider very general choices for the branch monodromies.

We will succeed following the above program in the following sense
\begin{itemize}
\item
we classify various sets of equivalence classes of $G$-covers by elements of
a monoid or a set with a monoid action.
\item
we classify certain sets of stable equivalence classes of $G$-covers by
elements of a set, which can be distinguished by the primary numerical type,
the Nielsen type and using $H_{2,\Gamma}$.
\item
we classify equivalence classes of $G$-covers by their stable equivalence
classes, if the Nielsen type is sufficiently large.
\end{itemize}

For example we prove the algebraic version, Thm.\ref{bijective} of the 
following geometric result.

\begin{thm}
\label{bijection}
Suppose $\Gamma$ is a union of conjugacy classes generating $G$, then
there exists an integer $N$ such that
\begin{quote}
the number of equivalence classes of $G$-covers of the disc
with local monodromies in $\Gamma$
and fixed Nielsen type $\nu$ is independent of $\nu$,
if $\nu$ takes a value at least $N$ on each conjugacy class.
\end{quote}
\end{thm}

Let us give a short overview on the content of this article:
In the next section we recall Hopf formula and explain
its validity in the setting of crossed modules. We then
take some care to motivate the definition of a $G$-crossed
module taking a union $\Gamma$ of conjugacy classes
into account, which gives rise to a finite abelian group
which is later shown to be equal to $H_{2,\Gamma}$.

In section 3 we equip the set of isomorphism classes of
$G$-covers with an algebraic structure of monoid.
This provides the necessary tool to investigate the geometric 
notion of stable equivalence by algebraic means.
In the following section we address the stabilisation of $G$-covers
branched over the disc, and determine conditions on the Nielsen types,
such that stabilisation is surjective, resp.\ bijective.

Section 5 considers various geometric generalisations of $G$-covers
and their equivalence classes. The corresponding algebraic setting
is then presented and explored.

In the final section 6 we revise the definition of tautological central extension
and recall the definition of quotient $H_{2,\Gamma}$ of $H_2(G)$ from
\cite{CLP15}.
We will then see how it fits very well with the set-up of the previous
sections and deduce a classification result for $G$-covers
with sufficiently large Nielsen type.


\section{algebraic setting}

The formula of Hopf describes the second homology group
$H_2(G,\ZZ)$ as the kernel of a natural map associated
to a given group $G$. To avoid unnecessary generality
we restrict to the case of a finite group of order $n:=\ord(G)$.
In that case there is a natural finite presentation of $G$
for any map $S\to G$, such that the image generates $G$
as a group. It is expressed in a short exact sequence
\[
1 \quad \to \quad R \quad \to \quad F \quad \to \quad G \quad \to\quad  1 \, ,
\]
where $F=F_S$ is the free group freely generated by elements
of $S$ and $R=R_S\subset F$ is the free subgroup of relations.

The Hopf formula \cite{Hopf} then states
\[
H_2(G,\ZZ) \quad = \quad \frac{ R \cap [F,F]}{[F,R]}.
\]
The same information is conveyed in the following exact sequences
\[
\begin{matrix}
1 & \to & H_2(G,\ZZ) & \to & \displaystyle \frac{F}{[F,R]}
& \to & \displaystyle\frac{F}{[F,F]} \times G & \to & 1
\end{matrix}
\]
\[
\begin{matrix}
1 & \to & H_2(G,\ZZ) & \to & \displaystyle \frac{R}{[F,R]}
& \to & \displaystyle \frac{F}{[F,F]}
& \to & \displaystyle\frac{F}{R[F,F]}
& \to \quad 1
\end{matrix}
\]

However, there is another setting which also provides an approach
to the second homology of $G$ using a free object, the category of
\emph{$G$-crossed modules}.
Let us quickly recall the basic definition.

\begin{defi}
A group homomorphism $\partial:C\to G$ with an action
\begin{eqnarray*}
C\times G & \to & C \\
(c,g) & \mapsto & c^g
\end{eqnarray*}
is called a \emph{$G$-crossed module}, if
\begin{description}
\item[XM1] $\partial$ is $G$-equivariant for the conjugation action of $G$
on itself,
\[
\partial(c^g) \quad = \quad (\partial c)^g  
\quad = \quad g^{-1}(\partial c)g,\qquad \forall\,c\in C,\:g\in G.
\]
\item[XM2] the \emph{Peiffer identities} hold
\[
c a \quad = \quad a c^{\partial a}, \qquad \forall\, a,c\in C.
\]
\end{description}
\end{defi}

\begin{exa}
If $S\subset G$ generates $G$ then
\[
\partial_S\,:\, \frac{F_S}{[F_S, R_S]} \quad \tto \quad G
\] 
is a $G$-crossed module which is a free $G$-crossed module on the inclusion 
$S\to G$, cf.\ \cite{ratcliffe}.
If $S$ is invariant under conjugation the $G$-action is simply induced by
\begin{eqnarray*}
G \times F_S & \tto & F_S \\
(g, \hat a) & \mapsto & \widehat{g^{-1}ag}
\end{eqnarray*}
which readily gives the $G$-equivariance of $\partial_S$:
\[
\partial_S (\hat a^g) \quad = \quad
\partial_S(\widehat{g^{-1}ag}) \quad = \quad g^{-1}ag \quad = \quad
g^{-1}(\partial_S a) g.
\]
\end{exa}

We note that the denominator of the Hopf formula has been incorporated into the crossed
module. So it remains to perform the intersection of its numerator, between a derived subgroup
and a kernel. This can be done with lots of $G$-crossed modules with still the second homology
of $G$ dropping out.

\begin{thm}[\cite{ep}]
If $\partial:C\to G$ is a projective $G$-crossed module ( in particular if it is a free $G$-crossed
module) with $\partial$ surjective, then
\[
H_2(G,\ZZ) \quad = \quad \ker \partial \cap [C,C].
\]
\end{thm}

We will now define a $G$-crossed module depending on a union
$\Gamma\subset G$ of conjugacy classes.
To this end, we first relax the notion of $G$-crossed module,
dispensing with the group structure of the domain.

\begin{defi}
A map $\varepsilon: Q\to G$ from a $G$-set $Q$ to the group $G$
is called an \emph{augmentation} and $Q$ an \emph{augmented quandle}
if
\begin{description}
\item[AQ1]
$\varepsilon$ is $G$-equivariant for the conjugation of $G$ on itself,
($Q$ is a $G$-crossed set)
\[
\varepsilon(p^g)\quad = \quad g^{-1} \varepsilon(p) g, \quad \forall\, p\in Q, g\in G
\]
\item[AQ2] idempotency holds in the sense
\[
p^{\varepsilon(p)} \quad = \quad p,\quad \forall\,p\in Q.
\]
\end{description}
\end{defi}

$G$-crossed modules are augmented quandles, but also, more
importantly, the set $\Gamma$.

\begin{exa}
The union of conjugacy classes $\Gamma\subset G$ is a $G$ set for the
action by conjugation. Therefore the injection $\varepsilon:\Gamma\to G$
is $G$-equivariant in the sense of {\bf AQ1}. Property {\bf AQ2} also holds
since any element in $G$ is unchanged under conjugation with itself.
\end{exa}

We do not want to recall the story of the notion of quandle here, but refer for
this and more on augmented quandles to the source \cite{joyce}.
Instead we take the quickest path back to $G$-crossed modules.
It leads via the following definition.

\newcommand{\adj}{\operatorname{Adj}}
\newcommand{\id}{\operatorname{Id}}


\begin{defi}
\label{adjoint}
Let $Q$ be an augmented quandle then the \emph{adjoint group} $\adj Q$
is the group presented by
\[
\adj Q \quad = \quad
\langle e_q, q\in Q\:|\: e_p e_q = e_q e_{p^{\varepsilon(q)}} \rangle.
\]
\end{defi}


It has a unique group homomorphism to $G$ compatible with the 
augmentation $\varepsilon$
\[
\partial_Q : \adj Q \to G, \quad e_q \mapsto \varepsilon(q)
\]
since $\adj Q$ has the universal property for quandle maps to groups.
And it serves our purpose, thanks to



\begin{prop}
\label{crossed quandle}
Suppose $\varepsilon:Q\to G$ is an augmented quandle, then
\[
\partial_Q:\,\adj Q \quad \to \quad G
\]
is a crossed module over $G$ with respect to the tautological
action of $G$ on $\adj Q$ induced by the action of $G$ on $Q$
\begin{eqnarray*}
\adj Q \times G & \to & \adj Q \\
(e_q, g) & \mapsto & {e_q}^g \quad = \quad e_{q^g}
\end{eqnarray*}
\end{prop}

\begin{proof}
First we check that $\partial_Q$ is $G$-equivariant:
\[
\partial_Q({e_q}^g) = \partial_Q(e_{q^g})  = \partial_Q(\varepsilon_Q(q^g)) 
= \varepsilon(q^g) = g^{-1}\varepsilon(q)g = \partial_Q(\varepsilon_Q(q))^g
= \partial_Q(e_q)^g.
\]
Second we have to check the Peiffer identities:
\[
e_p^{\partial_Q(e_q)} = e_p^{\partial_Q(\varepsilon_Q(q))}
= e_p^{\varepsilon(q)} = e_{p^{\varepsilon(q)}} = e_{p\ast q}
= e_q^{-1} e_p e_q
\]
\end{proof}

In analogy to the Hopf formula we get an abelian group
\begin{equation}
H(\G,G) \quad := \quad \ker \partial_\Gamma \cap [\adj \Gamma,\adj \Gamma].
\end{equation}

As we will see, that intersection
is a proper quotient of the second homology in
general.

\section{the Hurwitz monoid}

\newcommand{\teal}[1]{{\color{teal} #1}}

To associate to a $G$-cover of the punctured disc an algebraic object,
we recall the following geometric set-up:

Let $D$ be the closed unit disc, $(p_n)\subset D$ a sequence of
distinct points in the interior, $p_0\in\del D$ a base point and let
$(\gamma_n)$ be a geometric basis for
$\pi_1(D\setminus (p_n),p_0)$.


Any $G$-cover of $D$ -- connected or not -- which is unbranched outside
$p_1,\dots,p_d$ has a well-defined monodromy map relative $p_0$.
To the homotopy class of a closed path it assigns the unique bijection
acting from the right on the fibre at $p_0$ which gives the same map
as path-lifting. 

Only after choosing a point $\tilde p$ in the fibre, i.e.\ for a \emph{pointed
$G$-cover}, a monodromy map $\mu$ with values in $G$ is well-defined in 
general\footnote{points in the same orbit under the center of $G$ determine 
$G$-equivariantly isomorphic covers, in particular the choice is superfluous in case 
of abelian $G$.}:
Via the (left) $G$-action the fibre at $p_0$ is identified with $\{ g\tilde p|g\in G\}$
and $\mu$ takes the value $g$ if path-lifting maps $\tilde p$ to $g\tilde p$.
\begin{quote}
From now we tacitly assume all $G$-covers and their isomorphisms 
to be pointed and only
put an occasional (pointed) to remind the reader of this fact.
\end{quote}
In these circumstances, the chosen datum associates to a $G$-cover
a monodromy tuple of elements in $G$.
\[
\mu(\gamma_1), \dots, \mu(\gamma_d).
\]
As in previous papers, 
see \cite{CLP15,CLP16}, 
we want to call such a tuple a Hurwitz vector of genus $0$ according to
\begin{defi}
For any $d>1$ the set $G^d$ is a $\br_d-$set by the well-known
Hurwitz action of the braid group. An element of this $\br_d-$set
is called a \emph{$(0,d)$-Hurwitz vector}, while an element of 
$\coprod G^d$ is simply called a \emph{Hurwitz vector of genus $0$}.
\\
Two Hurwitz vectors $v, w$ are \emph{braid equivalent} if they belong
to the same braid group orbit,
\[
v\quad \approx \quad w.
\]
\end{defi}

Since on the geometric side
we only consider isomorphism classes of $G$-covers, i.e.\ up to
$G$-equivariant (pointed) maps covering a map preserving $(p_n)$ as a set,
we get a one to one correspondence
\[
\left\{  G\text{-covers of }D\text{, branched outside }p_1,\dots, p_d\right\}
/_\text{iso}
\quad \stackrel{1:1}{=\!=} \quad
G^d/\br_d\,.
\]

Now on the algebraic side we have a natural composition 
$G^d\times G^{e}\to G^{d+e}$ given by juxtaposition which
obviously is associative.
It is equivariant under the inclusions $\br_d\to\br_{d+e}$
and $\br_e\to \br_{d+e}$, where in the second case a braid
on strands one to $e$ is moved to the corresponding braid
on strands $d+1$ to $d+e$.
The resulting monoid is called the \emph{Hurwitz class monoid}
of $G$:
\[
\hfami_G \quad := \quad \coprod_{d\geq 0} G^d/\br_d
\]
where the unit is understood to be represented by the empty tuple. 

This composition can also be constructed on the geometric side.
Given a pair of $G$-covers unbranched outside $p_1,\dots, p_d$,
respectively $p_1,\dots, p_e$, there is a unique class of $G$-covers
unbranched outside $p_1,\dots, p_{d+e}$, which is isomorphic
to the first over a regular neighbourhood of 
$\gamma_1\cup\dots\cup\gamma_d$ and to the second over
a regular neighbourhood of $\gamma_{d+1}\cup\dots\cup\gamma_{d+e}$,
pointed by the same point $\tilde p$ over $p_0$.

On a moments thought, we may restrict on the geometric side to
$G$-covers with local monodromies in a union of conjugacy classes
$\Gamma\subset G$ and get on the algebraic side a submonoid
\[
\hfami_{G,\Gamma} \quad :=\quad \coprod_{d\geq0} \Gamma^d/\br_d.
\]

We next define equivalence of $G$-covers with respect
to stabilisation by a $G$-cover $C_u$ corresponding to some
$u\in G^e$.
\begin{defi}
Two $G$-covers 
are called \emph{$C_u-$stably equivalent} if
composition of either with the same number $\ell$ of copies of $C_u$
yield isomorphic $G$-covers.
\\
The corresponding elements $v, w\in\coprod G^d$ are called 
\emph{$u-$stably equivalent}
if $vu^\ell \approx wu^\ell$, for some $\ell$.
\end{defi}

So $u$-stable equivalence of $v, w$ is equality in the Hurwitz class
monoid $\hfami_G$ of elements represented by $vu^\ell, wu^\ell$ 
for some $\ell$. 
In fact we can push this further to an equality in some monoid of
fractions, but we need to recall the evaluation map on $\hfami_G$
first.

\begin{lemma}
\label{evaluation}
The \emph{evaluation map} defined on $\coprod G^d$ with values
in $G$
\[
v = (v_1,\dots,v_d) \quad \mapsto \quad ev(v):=v_1\cdots v_d \in G
\]
has the following properties:
\begin{enumerate}
\item
$ev$ is a monoid homomorphism.
\item
$ev$ is constant on braid group orbits.
\item
$ev$ induces an monoid homomorphism on the monoid $\hfami_G$,
also called \emph{evaluation}.
\end{enumerate}
\end{lemma}
\noindent
The proof is easy and left to the reader.

Now $ev(u)\in G$ has finite order -- say $n$ -- so we deduce that 
the elements $u^{\ell n}$ are central in $\hfami_G$ thanks to the
following lemma proved in \cite{CLP15}.

\begin{lemma}
\label{central}
If $v, w$ are Hurwitz vectors of genus $0$ and $ev(v)=1\in G$,
then $vw$ and $wv$ are braid equivalent.
\end{lemma}

In particular these elements form a central submonoid of $\hfami_G$,
which is the same as a central multiplicative set.
The following technical result will provide the existence of corresponding
monoids of fractions and expose their relation with the 
\emph{enveloping group} of $\hfami_G$, defined by the universal
property, that every monoid homomorphism to a group factors 
uniquely through the monoid homomorphism to the enveloping group.

\begin{prop}
\label{fraction_monoids}
Suppose $\sfami\subset\hfami_G$ is a submonoid, equal to $\hfami_G$
or central and generated by an element $u^\ell$ with $ev(u)$ of order 
dividing $\ell$,
then
\begin{enumerate}
\item
there is an equivalence relation on $\hfami_G\times\sfami$ given by
\[
(v,s) \quad \sim \quad (v',s') \qquad \iff \qquad
\exists \tilde v\in \hfami_G, \tilde s\in\sfami :
v\tilde r=v'\tilde s,\: s\tilde r=s'\tilde s
\]
Equivalence classes in case $\sfami$ generated by $u^\ell$ are 
written $[v/s]_u$.
\item
there is a well-defined monoid structure on equivalence classes
induced by multiplication in $\hfami_G$, which for $\sfami$ central is
\[
\begin{bmatrix} \frac v{u^{n\ell}} \end{bmatrix}_u
\begin{bmatrix} \frac w{u^{m\ell}} \end{bmatrix}_u
=\begin{bmatrix} \frac{vw}{u^{(n+m)\ell}} \end{bmatrix}_u.
\]
This monoid is denoted by $\hfami_G \sfami^{-1}$.
\item
the \emph{total monoid of fractions} $\hfami_G \hfami_G^{-1}$
together with the map is uniquely isomorphic
to the enveloping group of $\hfami_G$
\item
the monoid of fractions associated to $u$ is isomorphic to the
enveloping group if
\[
g\in G \quad \implies \quad \exists v\in \hfami_G, n>0:
gv = u^{n\ell} \in \hfami_G
\] 
\end{enumerate}
\end{prop}

\begin{proof}
The first two claims follow from the following two properties, well-know 
from localisation of rings, see \cite{Skoda}.
\begin{itemize}
\item
$s\in\sfami, v\in\hfami_G \implies s\hfami_G\cap v\sfami \neq\emptyset$
\item
$s\in\sfami, v,w\in\hfami_G, sv=sw \implies \exists s'\in\sfami: vs'=ws'$
\end{itemize}
Both are immediate for $\sfami$ central. The argument in the case 
$\sfami=\hfami_G$ is only slightly more difficult.
In fact we note that every element has a power that is central, since
its evaluation in the finite group $G$ has finite order and by Lemma
\ref{central}.
\begin{eqnarray*}
s\in\sfami, v\in\hfami_G & \implies &
sv^{\ord ev(v)} = v^{\ord ev(v)}s \in s\hfami_G\cap v\sfami
\\
s\in\sfami, v,w\in\hfami_G, sv=sw & \implies &
s^{\ord ev(s)}v=s^{\ord ev(s)}w
\\
& \implies & vs^{\ord ev(s)}=ws^{\ord ev(s)}
\end{eqnarray*}

To get the last two claims we first notice that for the monoids of
fractions $\hfami_G\sfami^{-1}$ by construction any monoid 
homomorphism to a group factors. Thus we only need to show
that every element has an inverse.
This is trivially true for the the total monoid of fraction. In the 
case of claim $iv)$ the elements of $\hfami_G$ given by a 
single letter $g\in G$ have an inverse $[v/u^{n\ell}]$ by
hypothesis. Since these elements generate $\hfami_G$
every element in the monoid has an inverse.
\end{proof}

We followed this route of progressive abstraction to have
a criterion for stably equivalence in terms of an equality
in a monoid of fraction which even is a group in the cases
we are mostly interested in:

\begin{cor}
\label{stab_eq1}
Suppose $u\in\hfami_G$ has $ev(u)\in G$ of order $\ell$,
then the following are equivalent
\begin{enumerate}
\item
$v,w\in \hfami_G$ are $u-$stably equivalent
\item
$[v/1]_u = [w/1]_u\:\in\, \hfami_G\sfami^{-1}$
\end{enumerate}
\end{cor}

Obviously, the proof of the proposition works also to proof
analogous claims for submonoids of the monoids $\hfami_{G,\G}$.
However we skip the details and only spell out the corresponding corollary.

\begin{cor}
\label{stab_eq2}
Suppose $u\in\hfami_{G,\G}$ has $ev(u)\in G$ of order $\ell$
and $\sfami=\{u^{n\ell}\}$,
then the following are equivalent
\begin{enumerate}
\item
$v,w\in \hfami_{G,\G}$ are $u-$stably equivalent
\item
$[v/1]_u = [w/1]_u\:\in\, \hfami_{G,\G}\sfami^{-1}$
\end{enumerate}
\end{cor}

In particular, if we want to decide $u$-stable equivalence in 
$\hfami_{G,\G}$ under the hypothesis that the monoid of fractions
is a group, then we can do so in the finitely presented group 
$\adj \Gamma$.

\begin{prop}
Suppose $u\in\hfami_{G,\G}$ has $ev(u)\in \G$ of order $\ell$
and $\sfami=\{u^{n\ell}\}$ is a denominator set such that
$\hfami_{G,\G} S^{-1}$ is a group, then
\[
\begin{split}
&v,w\in \Gamma^d\text{ are }u-\text{stably equivalent} \\
\iff\quad & e_{v_1}\cdots e_{v_d} \:=\: e_{w_1}\cdots e_{w_d}
\:\in\: \adj \Gamma
\end{split}
\]
\end{prop}

\begin{proof}
By the corollary above it suffices to show that the second claim
of the proposition is equivalent to the second claim of the corollary.
This follows since both groups involved are the
enveloping group of the monoid $\hfami_G$, and 
the element involved are mapped to
each other by the unique map provided by the universal property.
For the monoid of fractions this was shown above, for the adjoint
group it is proved by Kamada, Matsumoto \cite{KM}. It uses the fact
that the universal properties of the enveloping group and the 
adjoint group provide mutually inverse group homomorphisms
between the two.
\end{proof}


\section{the genus zero case}

We investigate the relation between equivalence and
stable equivalence of (pointed) $G$-covers
branched over the disc, i.e.\ with $g=0$.
They will be shown to be equal for the $G$-covers with sufficiently 'rich'
branching and detectable in monoids of fractions from the last section.
This involves studying conditions on the Hurwitz vector $u$ such that 
stability holds for geometric stabilisation by the corresponding
$G$-cover $C_u$ according to

\begin{defi}
We say that \emph{stability holds} for geometric stabilisation by
$C_u$, if there exists a positive integer $m$, such that
\[
C_u-\text{stable equivalence} \quad = \quad G-\text{cover equivalence}
\]
on the $G$-covers equivalent to covers obtained by $m$ iterations
of the $C_u$-stabilisation. 
\end{defi}
\medskip

The condition of the definition easily translates into a condition on
the algebraic side. Stability holds if 
\[
u-\text{stable equivalence} \quad = \quad
\text{braid equivalence}
\]
on the set of tuples braid equivalent to some $vu^m$.
We are thus bound to study properties of Hurwitz vectors of the last
kind in more detail. 

Before entering into the discussion of stability, let us review some important tools,
giving the definitions and providing their basis properties.

\begin{defi}
On the set $\coprod G^d$ of Hurwitz vectors for the finite group $G$
\begin{enumerate}
\item
the relation
\[
v\leq w \quad :\iff \quad \exists u\::\: vu=w
\]
is called the \emph{prefix order} on $\coprod G^d$.
\item
the relation
\[
v\lessapprox w \quad :\iff \quad \exists u\::\: vu\approx w
\]
is called the \emph{weak prefix order} on $\coprod G^d$,
and so is the induced order on $\hfami_G$.
\item
an element $v\in G^d$ is said to \emph{generate} the subgroup
$H\subset G$ if
\[
\langle v \rangle \quad := \quad \langle v_1,\dots, v_d\rangle \:=\: H.
\]
\end{enumerate}
\end{defi}
Now let $\{\cfami_i\}$ denote the set of equivalence classes for conjugacy on $G$,
then
\[
\ZZ\big(G\big/_{\!\!\sim}\big) \quad = \quad \oplus_i \:\ZZ\, \cfami_i
\quad \text{ with }\quad 
\nu\leq \nu' \quad :\iff \quad \nu_i \leq \nu_i',\: \forall i
\]
is a free abelian group with a partial order. Note that more generally
the free abelian group on a union $\Gamma$ of conjugacy classes of $G$ is
naturally isomorphic to the abelianisation of the adjoint group
for the conjugation quandle $\G$
\[
\oplus_{\cfami_i\subset\Gamma} \ZZ\,\cfami_i \quad = \quad
H_1(\adj(\Gamma),\ZZ).
\]
Next we recall the basic properties of the \emph{Nielsen type} and the
subgroup generated by a Hurwitz vector, but leave the proofs to the reader, see also \cite{CLP15,CLP16}.

\begin{lemma}
\label{Nielsen}
The \emph{Nielsen map} defined on $\coprod G^d$ with values
in $\oplus_{\cfami_i} \ZZ\,\cfami_i=\ZZ(G/_\sim)$, 
\begin{eqnarray*}
v = (v_1,\dots,v_d) & \mapsto & \nu(v):= \sum \nu_i \cfami_i 
\\
\nu_i & = & \# \{ j\:|\: v_j \in \cfami_i \}
\end{eqnarray*}
has the following properties:
\begin{enumerate}
\item
$\nu$ is a monoid homomorphism.
\item
$\nu$ is constant on braid group orbits.
\item
$\nu$ induces a monoid homomorphism on $\hfami_G$
also called \emph{Nielsen map}.
\item
the Nielsen map is order preserving for both prefix orders
\[
v \leq w \quad \implies \quad v\lessapprox w \quad \implies \quad \nu(v) \leq \nu(w),
\]
\end{enumerate}
\end{lemma}
\noindent

\begin{lemma}
\label{generator}
The map $\langle\,\,\rangle:\coprod G^d\to \{ H \,|\, H\subset G\}$ to the subgroups of 
$G$ partially ordered by inclusion 
has the following properties:
\begin{enumerate}
\item
$\langle\,\,\rangle$ is order preserving for both prefix orders on $\coprod G^d$.
\item
$\langle\,\,\rangle$ is constant on braid group orbits.
\item
$\langle\,\,\rangle$ induces an order preserving map on $\hfami_G$
with respect to the weak prefix order.
\end{enumerate}
\end{lemma}
\noindent

A special role will be given to the Hurwitz vector
defined in terms of $\Gamma\subset G$
\[
u_\Gamma = (\underbrace{g_1,\dots, g_1}_{\ord g_1},
\dots , \underbrace{g_r,\dots, g_r}_{\ord g_r}),
\]
where $g_1,\dots, g_r$ is an enumeration of the elements of
$\Gamma$. Its braid equivalence class is independent of choices
thanks to Lemma \ref{central}, since $ev(u_\Gamma)=1_G$.
The corresponding denominator set
generated by $u_\Gamma$ will be denoted by $\sfami_\Gamma$.

In the sequel, our argument is motivated by the approach of Conway and Parker,
see \cite{fv}.
The first step is the first part of
\cite[Lemma 6.9]{mm}.

\begin{lemma}
\label{conway}
Suppose $g_1,g_2$ are conjugate elements
of order $n$ in a finite group $G$ and $v$ is a Hurwitz vector which generates $G$,
then
\[
v g_1^n  \quad\approx\quad v g_2^n.
\]
\end{lemma}

The second part of that lemma inspired our next result:

\begin{lemma}
\label{conway_esk}
Let $w$ be a Hurwitz vectors generating $G$ 
and $u$ a Hurwitz vector with entries in $\Gamma$, then
\[
\nu(w)\geq \nu(u_\Gamma u) \quad \implies \quad
\exists v\::\: w\approx v u,\:v\text{ generates }G
\]
\end{lemma}

\begin{proof}
Thanks to induction, it suffices to prove the claim in case $u$ consists of a single
entry $g$. By hypothesis the conjugacy class of $g$ occurs so often that
by the pigeonhole principle there is a $g_1$ conjugate to $g$ which occurs
at least $n=\ord g +1$ times. By Hurwitz moves we can obtain a Hurwitz vector
$w' g_1^{n+1}$ braid equivalent to $w$, hence $w' g_1$ generates $G$.
We apply Lemma \ref{conway} and obtain $w\approx w' g_1 g^n$.
Thus our claim follows for $u=g$ and $v=w' g_1 g^{n-1}$.
\end{proof}

Now we are ready to prove a bunch of bijectivity results for certain map
induced by stabilisation maps.

\begin{lemma}
\label{surjective}
Let $u$ be a Hurwitz vector and
suppose $\nu_0\geq\nu(u_{\Gamma(u)})$, then for all $n\geq1$
\[
\begin{array}{rccc}
\cdot u : &
\bigg\{ [ v ]_\approx \,\left|\, \begin{aligned}
\nu(v)&=n\nu(u)+\nu_0 \\ \langle u\rangle&\subset\langle v\rangle
\end{aligned} \right.\bigg\}
& \tto &
\bigg\{ [ w ]_\approx \,\left|\, \begin{aligned}
\nu(w)&=(n+1)\nu(u)+\nu_0 \\ \langle u\rangle&\subset\langle w\rangle
\end{aligned} \right.\bigg\}
\\[14pt]
& [v]_\approx & \mapsto & [v u]_\approx
\end{array}
\]
is surjective and there exists $m=m(u,\nu_0)$, such that it is bijective for
$n\geq m$.
\end{lemma}

\begin{proof}
To prove surjectivity we want to apply Lemma \ref{conway_esk}.
Indeed, for $w$ with $[w]$ in the range, we let $G=\langle w\rangle$
and $\Gamma=\Gamma(u)\cap G$. Then $u$ has entries in $\Gamma$
and $\nu(w)\geq \nu(u)+\nu_0\geq\nu(u_\Gamma u)$, since 
$u_\Gamma \lessapprox u_{\Gamma(n)}$. Hence we get the conclusion,
that $[w]=[w'u]$ with $\langle w'\rangle = G\supset\langle u \rangle$.

Bijectivity follows from the fact, that an infinite chain of surjective
maps between finite sets must eventually stabilise.
\end{proof}

\begin{prop}
\label{invariably}
Suppose $u_\Gamma$ generates $G$ invariably 
$(i.e.\ \nu(u'_\Gamma)=\nu(u_\Gamma)\Rightarrow\langle u'_\G\rangle=G)$,
then there exists an integer $m=m(G)$ such that for
\begin{itemize}
\item
any $\nu_1\geq\nu_0\geq m\nu(u_\Gamma)$ with $\nu_0,\nu_1\in\nu(\coprod \G^d)$
\item
any Hurwitz vector $u$ with $\nu(u)+\nu_0=\nu_1$
\end{itemize}
juxtaposition of $u$ induces a bijective map
\[
\begin{array}{rccc}
\cdot u : &
\bigg\{ [ v ]_\approx \,\Big|\, \nu(v)=\nu_0 \bigg\}
& \tto &
\bigg\{ [ v ]_\approx \,\Big|\, \nu(v)=\nu_1 \bigg\}
\\[0pt]
& [w]_\approx & \mapsto & [w u]_\approx
\end{array}
\]
\end{prop}

\begin{proof}
Apply Lemma \ref{surjective} and get $m=m(u_\Gamma,\nu(u_\G))$
such that $u_\Gamma\cdot$ is bijective for $n\geq m$.
There are now Hurwitz vectors $u_0, u'$ such that
\[
\nu(u_0)+m\nu(u_\Gamma)=\nu_0,\quad
u_0uu'\approx u_\Gamma^k\text{ for some }k>0.
\]
Then we get a factorisation of the bijective map $u_\Gamma^{k}$:
\[
\big\{  \nu(v)=m\nu_\G \big\} \stackrel{\cdot u_0}{\tto}
\big\{  \nu(v)=\nu_0 \big\} \stackrel{\cdot u}{\tto}
\big\{  \nu(v)=\nu_1 \big\} \stackrel{\cdot u'}{\tto}
\big\{  \nu(v)=k\nu_\G \big\}
\]
Since $u_\Gamma$ generates $G$ invariantly, every element in the
given sets generates $G$, hence Lemma \ref{surjective}
gives surjectivity of all three maps.
Thus we conclude that each of the surjective maps, in particular
$\cdot u$ is bijective.
\end{proof}

\begin{thm}
\label{bijective}
Suppose $u_\Gamma$ generates $G$,
then there exists an integer $m=m(G)$ such that
\begin{enumerate}
\item
for
\begin{itemize}
\item
any $\nu_1\geq\nu_0\geq m\nu(u_\Gamma)$ 
with $\nu_0,\nu_1\in\nu(\coprod \G^d)=\NN( \Gamma/_\sim)$
\item
any subgroup $H$ of $G$,
\item
any Hurwitz vector $u\in\coprod H^d$ with $\nu(u)+\nu_0=\nu_1$
\end{itemize}
juxtaposition of $u$ induces a bijective map
\[
\begin{array}{rccc}
\cdot u : &
\bigg\{ [ v ]_\approx \,\left|\, \begin{aligned}
 \nu(v)&=\nu_0 \\ \langle v\rangle&=H
\end{aligned} \right.\bigg\}
& \tto &
\bigg\{ [ v ]_\approx \,\left|\, \begin{aligned}
\nu(v)&=\nu_1 \\ \langle v\rangle&=H
\end{aligned} \right.\bigg\}
\\[14pt]
& [w] & \mapsto & [w u]
\end{array}
\]
\item
for any $\nu_1,\nu_0\geq m\nu(u_\Gamma)$
with $\nu_0,\nu_1\in\nu(\coprod \G^d)$
\[
\#\bigg\{ [ v ]_\approx \,\Big|\, \nu(v)=\nu_0 \bigg\}
\quad = \quad
\#\bigg\{ [ v ]_\approx \,\Big|\, \nu(v)=\nu_1 \bigg\}
\]
\end{enumerate}
\end{thm}

\begin{proof}
Let $m$ be a common stability bound for the finitely many Hurwitz
vectors $u'_\Gamma$ with $\nu(u'_\Gamma)=\nu(u_\Gamma)$.
Then we can handle the claim of $i)$.

In case $\nu_0\not\in\nu(\coprod H^d)$, also $\nu_1$ is not contained,
so both sides are empty sets and the claim is trivially true.
Otherwise pick $u_0\in\coprod H^d$ with $\nu(u_0)=\nu_0$ and
$u'_\Gamma\in\coprod H^d$ with $\nu(u'_\Gamma)=\nu(u_\Gamma)$
a multiple of $u_{\Gamma\cap H}$.
Claim $i)$ then follows as in the previous proofs.

The second claim is an immediate corollary, since the sets involved
decompose over all possible subgroups $H$ of $G$ into the sets
of the first claim and shown there to be bijective.
\end{proof}

\begin{defi}
Call $m$ as in the theorem a 
\emph{$\Gamma$-stability bound}.
\end{defi}

\begin{prop}
Suppose $v, w$ are Hurwitz vectors generating $G$ with entries in $\Gamma$, 
and $\nu(v)\geq \nu(u_\Gamma^m)$,
where $m$ is a $\Gamma$-stability bound, then
\[
v \approx w \quad \iff \quad
[v/1]=[w/1] \in \hfami_{G,\G}\sfami^{-1}_\Gamma,
\]
\end{prop}

\begin{proof}
We have only the prove the reverse implication. The first hypotheses
on the right implies
\begin{eqnarray*}
& v\:\:\approx_{u} w & u_\Gamma\text{-stably} \\
\implies &  vu_\Gamma^\ell \approx wu_\Gamma^\ell & \text{for some }\ell \\
\implies&  v\approx w& \text{by theorem \ref{bijective}}
\end{eqnarray*}
\end{proof}

In Corollary \ref{stab_eq1} we showed
that elements with the same invariant
are stably equivalent. In this proposition we show that elements
with sufficiently high Nielsen class 
\emph{belong to the stable range}, ie.\
where stable equivalence implies equivalence. 
The existence of a stability range expresses the fact that 
stability holds.


\section{generalisations}

In this section we will aim for some generalisations going beyond the
case of $G$-covers of the disc branched at a finite set of points up to
$G$-equivariant covering isomorphisms.

Though we still want to understand stabilisation under 'geometric'
composition with a $G$-cover $C_u$ over the disc corresponding
to some $u\in G^e$, there are several different direction open
to generalisation:
\begin{itemize}
\item
include $G$-covers over surfaces of higher genus or higher number
of boundary components, one boundary component is needed at 
least to perform geometric composition.
\item
modify the notion of isomorphism on the base: restrict the induced maps
to preserve isotopy classes of appropriate geometric objects.
\item
modify the notion of isomorphism on the fibres: restrict to maps
preserving $G$-markings of a set of fibres.
\end{itemize}

To get the flavour of these generalisations, let us look at some examples
and the corresponding algebraic structure.

\begin{exa}
On the geometric side consider a Riemann surface 
$\Sigma=\Sigma_g^1$ of genus $g$ with one boundary component.
Let $p_0$ be a point on the boundary, $(p_n)\subset\Sigma$ be
a sequence of distinct points in the interior. Finally let
$\alpha_i,\beta_i$ corresponding to handles, and a sequence
$(\gamma_n)$ corresponding to the interior points be elements
of a geometric basis for $\pi_1(\Sigma\setminus(p_n),p_0)$

As before any (pointed) $G$-cover of $\Sigma$, which is unbranched outside
$p_1,\dots,p_d$ has a monodromy map that gives rise to the
monodromy tuple of elements in $G$
\[
\mu(\alpha_1),\mu(\beta_1),\dots,\mu(\alpha_g),\mu(\beta_g),
\mu(\gamma_1),\dots,\mu(\gamma_d),
\]
In the present case, such tuples are naturally acted on by the
mapping class group $\map_{g,d}^1:=\map(\Sigma_g^1)$, and we recall that elements
of $G^{2g+d}$ are called $(g,d)$-Hurwitz vectors, since the
Cartesian product is considered as a $\map_{g,d}^1$-set.
We get the one to one correspondence
\[
\left\{ G\text{-covers of }\Sigma\text{, branched outside }p_1,\dots,p_d\right\}
/_\text{iso}
\quad \stackrel{1:1}{=\!=} \quad
G^{2g+d}/\map_{g,d}^1
\]
\begin{displaymath}
	\centering
  \includegraphics[height=5cm]{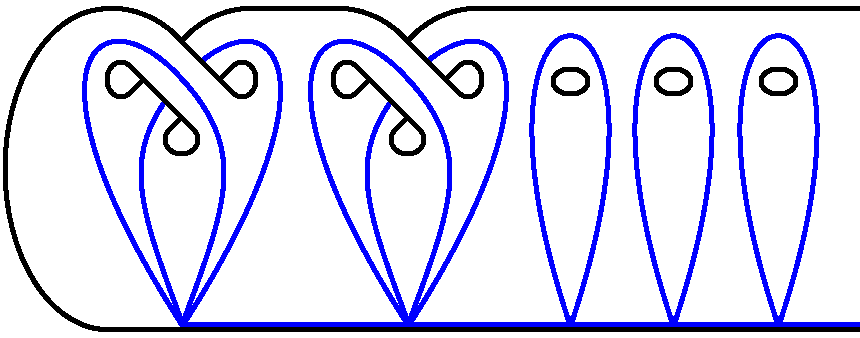}
  \linebreak
\end{displaymath}
\begin{center}base with geometric basis
\end{center}
\end{exa}

\begin{exa}
Let us consider the $G$-covers of $\Sigma$ as in the previous example.
However we define two $G$-covers to be \emph{isomorphic preserving
chains} if they are isomorphic via some $G$-equivariant covering map
such that the induced map preserves the free isotopy classes of the
$\alpha_i,\beta_i$.

Accordingly on the algebraic side we no longer have the action of the
full mapping class group but rather that of the subgroup preserving the
said isotopy classes. This subgroup can be identified with the mapping
class group $\map_{0,d}^{g+1}$ of a subsurface $\Sigma'$ of genus $0$
with $d$ punctures and $g+1$ boundary components obtained from
$\Sigma$ by cutting a regular neighbourhood of simple curves 
representing the given isotopy classes, thus
\[
\left\{ G\text{-covers of }\Sigma\text{, branched outside }p_1,\dots,p_d\right\}
/_{\text{iso}_{\alpha,\beta}}
\quad \stackrel{1:1}{=\!=} \quad
G^{2g+d}/\map^{g+1}_{0,d}
\]
\begin{displaymath}
	\centering
  \includegraphics[height=5cm]{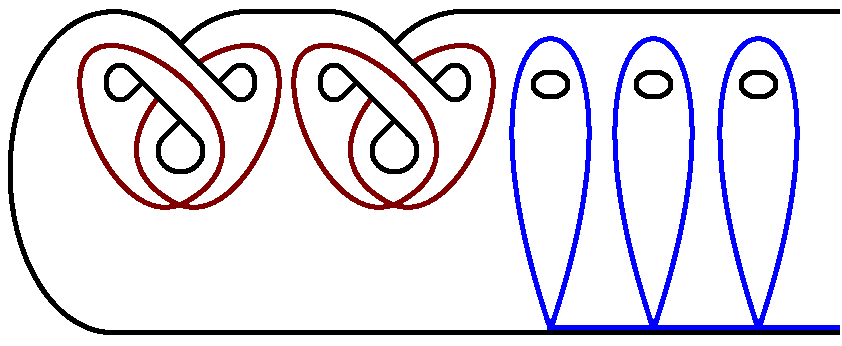}
  \linebreak
\end{displaymath}
\begin{center}base with curves representing two chains
\end{center}
\end{exa}

\begin{exa}
For the third kind of generalisation we look again at $G$-covers of the disc
$D$ equipped with distinct interior points $(p_n)$, a geometric basis 
$(\gamma_n)$ with respect to $p_0$ on the boundary.
Let $q_0,\dots,q_k$ be distinct points on the boundary with $q_0=p_0$.
We introduce a \emph{$q$-marking} to be given by $G$-equivariant maps
$\ell_i$ from $G$ to the fibres over the $q_i$.

Then we get an extended monodromy map defined on the fundamental
groupoid of the punctured base relative to the finite set of points $q_i$.
\[
\mu:\pi_1^{gr-oid}(D\setminus\{p_1,\dots,p_d\}, \{q_0,\dots,q_k\}) \:\tto\: G
\]
To the homotopy class of a path from $q_i$ to $q_{i'}$ 
it associates the unique element $g$ in $G$
such that $g\ell_{i'}(1)$ gives the same point in the fibre at $q_{i'}$
as path lifting to $\ell_{i}(1)$.

The domain is a free groupoid on $d+k$ generators, and we end up with the
one to one correspondence
\[
\left\{ \text{marked } G\text{-covers of }D\text{, branched outside }p_1,\dots,p_d\right\}
/_{\text{iso}_{\text{marked}}}
\quad \stackrel{1:1}{=\!=} \quad
G^{k+d}/\br_{d}
\]
The set $G^{k+d}$ in this case is a $\br_d$-set as the Cartesian product of
the trivial $\br_d$-set $G^k$ and the $\br_d$-set $G^d$ of Hurwitz vectors
of genus $0$.
\begin{displaymath}
	\centering
  \includegraphics[height=5cm]{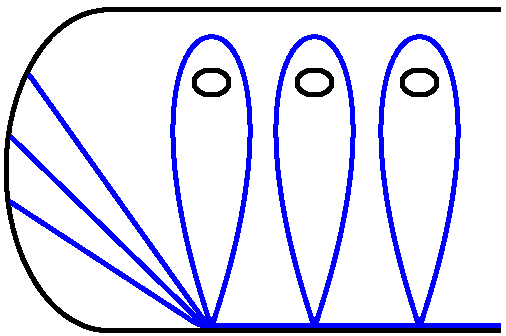}
  \linebreak
\end{displaymath}
\begin{center}base with geometric base of groupoid
\end{center}
\end{exa}
\medskip

The common features of these generalisations -- and others -- on the algebraic
side should be well noted:
\begin{itemize}
\item
we classify orbits of $G$-tuples,
\item
the action is by some geometrically distinguished mapping class group,
\item
$\br_d$ acts as the mapping class group of a suitable neighbourhood
of $\gamma_1,\dots,\gamma_d$.
\end{itemize}
They will allow  
to get an algebraic model for composition with
$G$-covers of genus $0$ on the geometric side,
which relies on the notion in the following definition:

\begin{defi}
Let $M$ be a monoid and $S$ be a set together with a map
\[
\rho: S \times M \to S, \quad
(s,m) \mapsto sm
\]
such that
\[
s(m_1m_2) = (sm_1)m_2,\quad s1_M = s
\]
Then we will call $\rho$ a \emph{$M$-action} and
$S$ a \emph{$M$-set} (instead of the more common \emph{$M$-act}).
\end{defi}

Let us formulate the hypothesis in more abstract terms

\begin{prop}
\label{HurwitzAct}
Suppose there is an array of nested groups 
\[
\begin{array}{cccccccccc}
\mfami_0 & \subset &  \mfami_1 & \subset &  \mfami_2  & \cdots & & \cdots & \mfami_d & \cdots
\\
& & \cup & & \cup & &&  & \cup &  \\
& & \br_1 & \subset & \br_2 & \cdots & & \cdots & \br_d & \cdots
\end{array}
\]
with $\mfami_d$ acting on $G^{r+d}$ such that
as a $\br_d$-set $G^{r+d}$ is the Cartesian product
of a trivial factor $G^r$ and the $\br_d$-set $G^d$ of Hurwitz vectors
of genus $0$, then\\[-2mm]
\begin{enumerate}
\item
$\coprod_d G^{r+d}$ is a $\coprod_d G^d$-set for concatenation,\\[-2mm]
\item
there is an induced action of the Hurwitz class monoid $\hfami_G$ on the
\emph{$\mfami$-orbits}
\[
\qfami_{G,\mfami} \quad := \quad \coprod_d G^{r+d}/\mfami_d
\]
\end{enumerate}
\end{prop}

\begin{proof}
Of course, the monoid $\coprod G^d$ is a $\coprod G^d$-set 
and the action is by concatenation. Thus the first claim follows
from the observation that $\coprod_d G^{r+d}$ is an invariant 
subset.

For the second claim we have to show that the action on 
equivalence classes does not depend on the representatives.
In fact
\begin{eqnarray*}
x\mfami_e = y\mfami_e, v\br_d = w\br_d & \implies &
xv \mfami_{e+d} = yw\mfami_{e+d}
\end{eqnarray*}
since the action of $\br_d$ on $G^d$ is the same as that on
the second factor of $G^{r+e+d}=G^{r+e}\times G^d$ via the action
of $\br_{e+d}\subset\mfami_{e+d}$ provided $\br_d$ is considered
as the subgroup of $\br_{e+d}$ braiding only the last $d$ strands.
\end{proof}

There is still one more essential feature:
\begin{itemize}
\item
the number of local monodromies in each conjugacy class
is invariant under the equivalence.
\end{itemize}
In fact a choice of a geometric free basis can be made in 
such a way, that the local monodromies correspond to the
entries of a tuple in $G^{r+d}$ except for the first $s$ entries,
with $s\leq r$.
Hence in extension of the Nielsen type we get

\begin{defi}
The \emph{Nielsen map} $\nu$ is defined on $G^{r+d}$ with values
in $\oplus_{\cfami_i} \ZZ\,\cfami_i=\ZZ(G/_\sim)$, 
\begin{eqnarray*}
v = (v_1,\dots,v_{r+d}) & \mapsto & \nu(v):= \sum \nu_i \cfami_i 
\\
\nu_i & = & \# \{ j>s\:|\: v_j \in \cfami_i \}
\end{eqnarray*}
\end{defi}

At this point we have collected enough information to turn back
to the previous sections and see that their results generalise all
along the way we have gone.

\begin{rem}
The action of $\hfami_G$ on $\qfami_{G,\mfami}$ can be 'localised'
at a denominator set $\sfami$ to yield a group action of the
group $\hfami_G\sfami^{-1}$ on the set $\qfami_{G,\mfami}\sfami^{-1}$.

Moreover
if $u\in\hfami_{G}$ has $ev(u)\in G$ of order $\ell$
and $\sfami=\{u^{n\ell}\}$,
then the following are equivalent
\begin{enumerate}
\item
$v,w\in \qfami_{G,\mfami}$ are $u-$stably equivalent
\item
$[v/1]_u = [w/1]_u\:\in\, \qfami_{G,\mfami}\sfami^{-1}$
\end{enumerate}
\end{rem}

We will say more about stabilisation. However, instead of pushing the generalisation
to the limits, we fix $\Gamma=\Gamma_G=G\setminus\{1_G\}$ which contains all non-trivial
elements of $G$ and generates $G$ invariably.

\begin{lemma}
\label{surjective_gen}
Let $u$ be a Hurwitz vector and
suppose $\nu_0\geq(r-s+1)\nu_{\Gamma}$, then
\[
\begin{array}{rccc}
&
\bigg\{ v \mfami \,\Big|\, 
\nu(v)=n\nu(u)+\nu_0 
\bigg\}
& \stackrel{\cdot u}\tto &
\bigg\{  w \mfami \,\Big|\, 
\nu(w)=(n+1)\nu(u)+\nu_0 
\bigg\}
\\[14pt]
& v \mfami & \mapsto & v u \mfami
\end{array}
\]
is surjective for $n>0$ and there exists $m=m(u,\nu_0)$, such that it is bijective for
$n\geq m$.
\end{lemma}

\newcommand{\tail}{{}^\tau\!}

\begin{proof}
The larger bound for $\nu_0$ is needed to make sure that for $w\in G^{r+d}$ its
\emph{tail} $\tail w\in G^d$ has $\nu(\tail w)\geq \nu(u)+\nu_\Gamma$. Then
we proceed as in the proof of Lemma \ref{surjective}
using Lemma \ref{conway_esk} and that every
Hurwitz vector with $\nu\geq\nu_\Gamma$ generates $G$ by our general
assumption $\Gamma=\Gamma_G$.
\end{proof}

\begin{thm}
\label{bijective_gen}
Let $\Gamma=\Gamma_G$, 
then there exists an integer $m=m(G)$ such that for
\begin{itemize}
\item
any $\nu_1\geq\nu_0\geq m\nu(u_\Gamma)$ 
\item
any Hurwitz vector $u$ with $\nu(u)+\nu_0=\nu_1$
\end{itemize}
juxtaposition of $u$ induces a bijective map
\[
\begin{array}{rccc}
&
\bigg\{ [ v ]_\mfami \,\Big|\, \nu(v)=\nu_0 \bigg\}
& \stackrel{\cdot u}\tto &
\bigg\{ [ v ]_\mfami \,\Big|\, \nu(v)=\nu_1 \bigg\}
\\[0pt]
& v \mfami & \mapsto & v u\mfami
\end{array}
\]
\end{thm}

\begin{proof}
Apply Lemma \ref{surjective_gen} with $u=u_\Gamma$ 
and $\nu_0=(r+1)\nu_\Gamma$ to get $m$ 
and let $m=m(u_\Gamma,\nu(u_\G))+r-s+1$
such that $u_\Gamma\cdot$ is bijective for $n\geq m$ on
\[
\bigg\{ v \mfami \,\Big|\, \nu(v)=n\nu_\Gamma \bigg\}
\]
There are now Hurwitz vectors $u_0, u'$ such that
\[
\nu(u_0)+m\nu(u_\Gamma)=\nu_0,\quad
u_0uu'\approx u_\Gamma^k\text{ for some }k>0.
\]
Then we can conclude as in the proof of Prop.\ \ref{invariably}.
\end{proof}

This is the important task, since the analogous result
for genus stabilisation will be applied in the argument
for homological stability.


\newcommand{\RGamma}{R_\Gamma}

 \section{The tautological lift}
 
In the final section we revise the definition of tautological central extension
and recall some information, see \cite{CLP15}, in particular the definition
and properties of a quotient $H_{2,\Gamma}$ of $H_2(G)$ which proved
to be crucial in our classification of curves with dihedral group of automorphisms.

We will then see how to recover that group in the set-up of the previous
sections and deduce a classification result for (pointed) $G$-covers 
in the stable range.

\begin{defi}
\label{GGamma}
Let $G$ be a finite group and let $F=F_G$, $R=R_G$ be as before. 
For any union of conjugacy classes $\G \subset G$,  define
\begin{enumerate}
\item
the \emph{tautological lift} $G\to F_G$, which maps $g \mapsto \hat g$,
\item
the \emph{tautological map} on $\coprod G^d: v=(v_1,\dots,v_d)\mapsto \hat v_1\cdots\hat v_d$
\item
the normal subgroup $\RGamma$ normally generated by commutators
in $[F,R]$ and tautological lifts of conjugacy relations for elements in $\Gamma$: 
\[
\RGamma \quad = \quad 
\big\langle\!\big \langle\; [F,R], 
\hat{a}\hat{b}\hat{c}^{-1}\hat{b}^{-1}\, \big| \, 
\forall a \in \G, b\in G, c=b\inv a b \;\big\rangle\! \big\rangle 
\]
( Note that the given elements generate this normal subgroup as subgroup. )
\item
the quotient group of $F$ by $\RGamma$
\[
G_\G \quad = \quad \fracc{F}{\RGamma} \, .
\]
\item
the \emph{boundary homomorphism} $\alpha: G_\Gamma\to G$, induced by $\hat a\mapsto a$ with kernel
$K_\Gamma$.
\end{enumerate}
\end{defi}


\begin{lemma}
\label{centr}
With the notation just introduced, 
$\RGamma \subset R$ and $K_\G=\fracc{R}{\RGamma}$.
In particular $K_\G$ is contained in the centre of $G_\G$ 
and the short exact sequence
$$
1 \to \frac{R}{\RGamma} \to G_\G \to G \to 1\, 
$$
is a central extension.
\end{lemma}

\begin{proof}
$[F,R]\subset R$ because $R$ is normal in $F$. 
Moreover $\hat{a}\hat{b}\hat{c}^{-1}\hat{b}^{-1}\in R$
for any $a,b,c\in G$ with $ab=bc$, therefore $\RGamma \subset R$. 
By the definition of $\a$ we have that $K_\G=\frac{R}{\RGamma}$. 
Finally, $K_\G$ is in the centre of $G_\G$ because  $[F,R]\subset \RGamma$.
\end{proof}

The \textit{tautological lift} $G\to G_\G$, $a\mapsto \hat{a}$
is not a group homomorphism in general, but every element in  
$G_\G$ with image $g\in G$ can be written as
$\hat{g} z = z \hat{g}$, with $z\in K_\G$.
Here, by abuse of notation, $\hat{g}$ denotes also the class of 
$\hat{g}\in F$ in $G_\G = \fracc{F}{\RGamma}$.

%

\begin{rem}
Let $\Gamma\subset G$ be the union of distinct conjugacy classes
$C_1,\dots, C_t$ and let $g_1,\dots,g_r$ be the 
elements of $G\setminus \Gamma$,
then the abelianisation $G_{\G}^{ab}$ of $G_{\G}$ 
is the free abelian group on $t+r$ generators
\[
G_{\G}^{ab} \quad\cong\quad
\ZZ C_1\oplus \dots \oplus \ZZ C_t \oplus \ZZ g_1 \oplus
\dots \oplus \ZZ g_{r}.
\]

The Nielsen map $\nu$ on $\coprod \Gamma^d$ factors 
through the tautological map and the abelianisation as
\[
\mbox{$\coprod_d$}\Gamma^d \quad \tto \quad
G_\Gamma \quad \tto \quad
G_\Gamma^{ab} \quad \tto \quad
\oplus_i \ZZ \cfami_i
\]
\end{rem}

\begin{defi}\label{H2Gamma}
Let $\G \subset G$ be a union of non-trivial conjugacy classes of $G$. We define
$$
H_{2,\G}(G)=\ker \left( G_\G \to G\times G_\G^{ab} \right) \, ,
$$
where $G_\G \to G\times G_\G^{ab}$ is the morphism with first component
the boundary map $\a$ and second component the abelianisation.
\end{defi}

Let us recall from \cite{CLP15} the precise relation between $H_2(G,\ZZ)$ and $H_{2,\G}(G)$.

\begin{lemma}\label{H2toH2Gamma}
Let  $G$ be a finite group and let  $\G \subset G$ be a union of nontrivial conjugacy classes.
Write $G=\frac{F}{R}$ and  $G_\G =\frac{F}{R_\G}$.
Then,  there is a short exact sequence
$$
1\to \frac{R_\G \cap [F,F]}{[F,R]}\to H_2(G,\ZZ)\to H_{2,\G}(G)\to 1 \, .
$$
In particular $H_{2,\G}(G)$ is abelian.
\end{lemma}




\begin{rem}
The Schur multiplier is often interpreted as a cohomology group
\[
H^2(G, \fracc{\QQ}{\ZZ})
\]
which is algebraically dual to $H_2(G,\ZZ)$. In case $\Gamma=G$ 
Moravec \cite{Mor11} identified the
group $H_{2,G}(G)$ with a subgroup of $H^2(G, \fracc{\QQ}{\ZZ})$
introduced by Bogomolov and justly called Bogomolov multiplier
by Moravec.
\end{rem}

The algebraic object from the previous to enter the stage is the adjoint
group associated a quandle from Def.\ref{adjoint}. To emphasise its importance for the 
pair $G,\Gamma$ we give a more specific name:

\begin{defi}
Suppose $G$ is a finite group and $\Gamma\subset G$ a union
of conjugacy classes, then the adjoint group for the quandle $\Gamma$
\[
\adj \Gamma \quad := \quad \langle e_g, g\in \Gamma \:|\: 
e_a e_b = e_b e_{a^b}, a,b\in \G\rangle
\]
is called the \emph{tautological crossed module} of the pair $G,\Gamma$.
The structural maps are
\[
\partial_\Gamma: e_g\mapsto g,\quad
\adj \Gamma \times G \to \adj \Gamma: (e_a, g) \mapsto e_{a^g}.
\]
(It is a crossed module thanks to Prop.\ref{crossed quandle}.)
\end{defi}

The two following results establish the close relation between the
'old' tautological central extension and the 'new' tautological
crossed module.

\begin{prop}
\label{compare1}
The map $e_a\mapsto\hat a$ for all $a\in\Gamma$ extends
to an injective group homomorphism
\[
\adj \Gamma \quad \tto \quad F/R_\Gamma
\]
with left inverse.
\end{prop}

\begin{proof}
For the extension we only need to check that the relations of the domain
map to $R_\Gamma$:
\[
e_a e_be_{a^b}^{-1}e_{b^{-1}}
\quad \mapsto \quad
\hat a \hat b {\hat c}^{-1} \hat{b}^{-1}, \text{ with }c = a^b=b^{-1}ab.
\]
For the left inverse we first define a map $F\to \adj \Gamma$ by
\[
\hat h \quad \mapsto \quad e_{g_1} \cdots e_{g_r}
\]
where for each $h\in G$ a unique factorisation $h=g_1\cdots g_r, r=r(h)$ 
has been chosen, with $r=1, g_1=h$ if $h\in\Gamma$.
Since the composition of the two maps is the identity on $\Gamma$, it
induces the identity map on $\adj\Gamma$ if $R_\Gamma$ is in the kernel
of this second map.

Let us note first, that by construction the map $F\to G$ factors through 
$\partial_\Gamma$. Accordingly $R$ maps to $\ker\partial_\Gamma$.
Now recall that for any crossed module $\partial :C\to G$ the kernel
is central in $C$. In fact, conjugation by any element in the kernel is 
trivial by the Peiffer identity. 
We infer, that $R$ maps to the centre of $\adj\Gamma$ and
hence $[F,R]$ maps to the identity.

We complete the proof by showing that this is true also for the 
remaining elements generating $R_\Gamma$:
\[
\begin{split}
\hat g \hat h \quad & \mapsto \quad e_g e_{g_1} \cdots e_{g_r}
= e_{g_1} e_{g^{g_1}} e_{g_2} \cdots e_{g_r}
= e_{g_1} \cdots e_{g_r} e_{g^h}
\\
\hat h \widehat{h^{-1}gh} \quad &\mapsto\quad
e_{g_1} \cdots e_{g_r} e_{g^h}
\end{split}
\]\\[-28pt]
\end{proof}

\begin{prop}
\label{kernels}
If $\Gamma$ generates $G$, then
\[
H_{2,\Gamma} \quad = \quad H(\Gamma, G)
\: := \: \ker\partial_\Gamma\cap[\adj \G,\adj\G]
\]
\end{prop}

\begin{proof}
Since the maps from the previous proposition induce the identity
on the quotient $G$ and maps on the respective abelianisations,
they also induce maps on the given groups.
Therefore the group of the right is a subgroup of the other and
it remains to show, that the kernel $N$ of $G_\Gamma\to \adj \Gamma$
intersects $H_{2,\Gamma}$ trivially.

By the construction of the map, $N$ is normally generated by
elements
\[
\hat g_1\dots \hat g_r \hat h^{-1}, \quad \text{ for all } h\not\in\Gamma
\]
They all map to $1\in G$, so they are central and generate a free abelian
subgroup in $G_\Gamma$ of rank $|G\setminus \Gamma|$.
The image in $\ZZ (G\setminus \Gamma)$ via $G_\Gamma^{ab}$
is of the same rank, since each generator is mapped to a standard 
generator. Thus $N$ maps injectively to $G_\Gamma^{ab}$ and
thus intersects trivially with $H_{2,\Gamma}$.
\end{proof}

To pursue the proof of the following classification result,
we first need to introduce another Nielsen map.

\begin{defi}
The canonical extension of the Nielsen map $\hfami_G\to \oplus_i \ZZ\cfami_i$
to the localisation
\begin{eqnarray*}
\nu: \hfami_G\sfami_\Gamma^{-1} & \tto & \oplus_i \ZZ\cfami_i
\\[0pt]
[v/u_\Gamma^\ell] & \mapsto & \nu(v) - \ell\nu(u_\Gamma)
\end{eqnarray*}
is also called \emph{Nielsen map}.
\end{defi}

\begin{thm}
\label{torsor}
Suppose $u_\Gamma$ generates $G$,
then there exists an integer $m=m(G)$ such that
for
any $\nu_0\geq m\nu(u_\Gamma)$ 
with $\nu_0\in\nu(\coprod \G^d)$
\[
\begin{array}{rccc}
&
\# \bigg\{ [ v ]_\approx \,\left|\, \begin{aligned}
 \nu(v)&=\nu_0 \\ \langle v\rangle&=G
\end{aligned} \right.\bigg\}
& = &
\# H_{2,\Gamma}\times [G,G]\\[14pt]
\end{array}
\]
\end{thm}

\begin{proof}
First we note that there is an exact sequence induced from the central
extension in Lemma \ref{centr}
\[
H_{2,\Gamma} \quad \tto \quad \ker \nu \quad \tto \quad
[G,G]
\]
Moreover, any $[v/u_\Gamma^\ell]\in \ker\nu$ gives two maps
\[
\begin{array}{rccc}
\cdot u_\Gamma, \cdot v: &
\bigg\{ [ v ]_\approx \,\left|\, \begin{aligned}
 \nu(v)&=\nu_0 \\ \langle v\rangle&=G
\end{aligned} \right.\bigg\}
& \tto &
\bigg\{ [ v ]_\approx \,\left|\, \begin{aligned}
\nu(v)&=\nu_0+\ell\nu_\Gamma \\ \langle v\rangle&=G
\end{aligned} \right.\bigg\}
\end{array}
\]
Both are bijective since we are above the stability bound, hence
we can define a well-defined action 
\[
[w]_\approx [v/u_\Gamma^\ell] \quad = \quad [w'], \quad
\text{such that } wv \approx w'u_\Gamma^\ell.
\]
Again using the fact, that we are in the stable range, the set
is mapped injectively to the enveloping group $\adj \Gamma$.
This can be exploited to show that the action is free and transitive,
because the action is now identified with multiplication inside 
$\adj \Gamma$ by the subgroup $\ker\nu$.

While freeness is immediate, we are left to check, that any two
elements $w,w'$ are in one orbit:
There is a unique element in the group $\adj \Gamma$ which
maps one to the other, which again we can write as
$[v/u_\Gamma^\ell]$. Under the Nielsen map, it must map to $0$,
since $\nu$ is a homomorphism on $\adj \Gamma$. Thus
this element is in $\ker \nu$ and we have also proved transitivity.
The claim on the cardinality of the set is then obvious.
\end{proof}


Rephrased in more geometrical terms the statement of the theorem tells us:
\begin{quote}
In the stable range connected (pointed) $G$-covers are classified up to 
equivalence by the Nielsen type, the evaluation, and an element in $H_{2,\Gamma}$.
\end{quote}
Do not miss the caveat: the homological information
is not canonical, but depends on the choice of an element
in each fibre of $\adj \Gamma \to G\times \oplus_i\ZZ \cfami_i$.



\end{document}